\documentclass[a4paper,french,oneside,11pt]{amsart}

\usepackage[utf8]{inputenc}
\usepackage[T1]{fontenc}
\usepackage[hscale=0.7, vscale=0.7]{geometry}
\usepackage[numbers]{natbib}

\usepackage{babel}
\frenchsetup{og=«,fg=»}

\usepackage{amssymb}

\usepackage{mathtools}

\DeclareFontFamily{OMX}{mlmex}{}
\DeclareFontShape{OMX}{mlmex}{m}{n}{%
   <->mlmex10%
   }{}%
\usepackage{mlmodern}

\theoremstyle{plain}
\newtheorem{thm}{Théorème}
\newtheorem{prop}[thm]{Proposition}

\theoremstyle{definition}

\newtheorem{rem}{Remarque}

\usepackage{color}

\usepackage[np,autolanguage]{numprint}

\usepackage{hyperref}
\hypersetup{%
  colorlinks=true,%
  citecolor=[RGB]{120,29,126},
  pdfauthor={Jean-François Burnol},%
  pdfsubject={Kempner series},%
  pdfstartview=FitH,%
  pdfpagemode=UseNone,%
}

\newcommand\NN{\mathbb{N}}

\def\leq{\leqslant}
\def\geq{\geqslant}

\allowdisplaybreaks

\usepackage{setspace}

\title[Asymptotique des sommes de Kempner]
  {Sur l'asymptotique des sommes de Kempner pour de grandes bases}

\author[J.-F. Burnol]{Jean-François Burnol}

\address{Université de Lille, Faculté des Sciences et technologies,
  Département de mathématiques, Cité Scientifique, F-59655 Villeneuve d'Ascq
  cedex, France} 

\email{jean-francois.burnol@univ-lille.fr} 

\date{May 20, 2024. This version differs from v2 at \url{https://arxiv.org/abs/2403.01957}
  only via a largely rewritten introduction, added references, the use of ``one-half''
  line spacing and a change of typeface.}

\subjclass[2020]{Primary 11Y60, 11M06; Secondary 11A63, 11B37, 30C10, 41A60, 44A60;}
\keywords{Ellipsephic numbers, Kempner series, asymptotic expansions}

\begin{document}

\singlespacing

\begin{abstract}
  Soit $K$ la somme des inverses des entiers n'ayant pas le chiffre $b-1$ en
  base $b$.  Nous montrons $K = b\log(b) - A/b - B/b^2-C/b^3+O(1/b^4)$ avec
  $A=\zeta(2)/2$, $B = (3\zeta(2)+\zeta(3))/3$ et $C =
    (2\zeta(2)+4\zeta(3)+\zeta(4))/4$.
\vspace*{\baselineskip}

    \begin{otherlanguage}{english}
\noindent      \textsl{English summary.}  Let $K$ be the sum of the reciprocals of
      positive integers which are missing the digit $b-1$ in their $b$-ary
      representation.  We prove $K = b\log(b) - A/b - B/b^2-C/b^3+O(1/b^4)$
      where $A=\zeta(2)/2$, $B = (3\zeta(2)+\zeta(3))/3$ and $C =
      (2\zeta(2)+4\zeta(3)+\zeta(4))/4$.
    \end{otherlanguage}
\end{abstract}

\maketitle

\onehalfspacing

\section{Théorème principal}

Kempner \cite{kempner} a remarqué en 1914 que la sous-série de la série
harmonique obtenue en ne conservant que les dénominateurs n'ayant aucune
occurrence du chiffre $9$ en base $10$ était convergente.  En effet il n'y a pour
chaque $\ell\geq1$ que $8\cdot 9^{\ell-1}$ termes $m^{-1}$ avec
$10^{\ell-1}\leq m<10^{\ell}$.
%
%
Notons $K(b,d)$ la sous-série de la série harmonique obéissant à la
restriction de Kempner, pour une base $b\geq2$ générale et un chiffre
$d\in\{0,\dots,b-1\}$.

Fine  \cite{segalleppfine1970} et d'autres auteurs, répondant à une question de
Segal et Lepp, ont montré $K(b,0) = b\log(b) + O(b^{-1})$.  Par ailleurs 
Kl{\o}ve \cite{klove1971} a obtenu $K(b,d)=b\log(b)+O(b)$ pour tout $d$.%
\footnote{Dans son article, Kl{\o}ve considère plus généralement des
  restrictions sur $m\geq1$ chiffres.  Comme nous l'avons montré dans
  \cite{burnoldigamma} (et vérifié numériquement), le terme principal est
  $b\log(b)/m$, et non pas $b\log(b)$.  Le résultat de Kl{\o}ve n'est correct
  que pour $m=1$.}
Ces résultats ont été améliorés par l'auteur dans \cite{burnolasymptotic} qui
montre pour tout $d>0$ fixé $K(b,d)=b\log(b)-b\log(1+d^{-1})+O_d(b^{-1})$ et
donne cinq termes supplémentaires du développement en puissances inverses de
$b$, pour tout $d>0$ ou $d=0$ fixé indépendamment de la base $b$.

D'autres régimes asymptotiques mènent au terme principal $b\log(b)$: Farhi
\cite{farhi} a montré $\lim_{k\to\infty} I(b,d,k) = b\log(b)$ pour tout $b$ et
$d$.  Ici $I(b,d,k)$ est la série harmonique avec des dénominateurs ayant
exactement $k$ fois le chiffre $d$ en base $b$, de sorte que $K(b,d) =
I(b,d,0)$.  De telles séries avaient été définies par Irwin \cite{irwin} peu
après l'observation de Kempner.  Allouche, Hu et Morin
\cite{allouchehumorin2024} ont considéré plus généralemeent la contrainte
qu'un bloc $w$ de $|w|$ chiffres apparaisse exactement $k$ fois.  Ils ont
montré que la limite pour $k$ tendant vers l'infini est alors
$b^{|w|}\log(b)$.  L'auteur a ensuite obtenu ce même résultat par une méthode
complètement différente \cite{burnolblocks}.  La vérification numérique n'est
pas facile (le calcul direct de la somme des séries est pour presque tous les
cas impossible à cause de la lenteur de la convergence) et les algorithmes
pour ce-faire ne sont disponibles dans la littérature que pour des cas
particuliers (\cite{baillie1979}, \cite{schmelzerbaillie}, \cite{baillie2008},
\cite{burnolone42}).

Les entiers ne possédant que certains chiffres dans une base $b$ donnée
(entiers «ellipséphiques» suivant une terminologie de C.~Mauduit) ont
fait l'objet de travaux avancés en théorie analytique des nombres, et Maynard
a montré qu'il y a toujours une infinité de nombres premiers de cette forme.
Nous n'abordons nullement ces sujets ici mais citons
\cite{erdosmauduitsarkozy1998}, \cite{dartygemauduit2000},
\cite{maynardInventiones2019} pour l'intérêt du lecteur.

Dans le présent article nous étudions la généralisation directe de la série
originelle de Kempner aux couples $(b,b-1)$ plutôt que $(10,9)$.  Voici notre
résultat:
\begin{thm}\label{thm:1}
  Soit $b>1$ un entier et $K(b,b-1)$ la somme des inverses $m^{-1}$, portant sur
  tous les entiers strictement positifs ne contenant dans leur écriture en
  base $b$ aucune occurrence du chiffre $b-1$.

  Alors, pour $b$ tendant vers l'infini:
\begin{equation}\label{eq:main}
  K(b,b-1) = b \log(b) 
- \frac{\zeta(2)}{2 b} 
- \frac{3\zeta(2)+\zeta(3)}{3b^{2}} 
- \frac{2\zeta(2)+4\zeta(3)+\zeta(4)}{4b^3}
+  \frac{O(1)}{b^{4}}
\end{equation}
\end{thm}

Ce n'est pas la première fois que les $\zeta(n)$ apparaissent dans
la théorie des séries de Kempner.  En effet
Fischer \cite{fischer} a démontré en 1993 la formule suivante:
\begin{equation}\label{eq:fischer}
  K(10,9) = 10 \log(10) - \sum_{m=1}^\infty 10^{-m-1} \beta_{m} \zeta(m+1)
\end{equation}
Dans cette formule les coefficients $\beta_m$ sont des nombres rationnels
vérifient les récurrences linéaires suivantes pour $m\geq1$ (en particulier
$\beta_0=10$):
\begin{equation}\label{eq:fischerrec}
  \sum_{k=1}^m \binom{m}{k}\bigl(10^{m-k+1}-10^k + 1\bigr)\beta_{m-k} = 10\,(11^m -10^m)\;.
\end{equation}

Comme nous l'expliquons dans un travail \cite{burnoldigamma} rédigé
postérieurement à celui-ci les coefficients $\beta_m$ sont des cas
particuliers des quantités dénotées $v_m$, fractions rationnelles en la
variable $b$, qui ont été définies par l'auteur dans \cite{burnolirwin}, pour
toute base et tout ensemble de chiffres exclus.  Et il est possible d'obtenir
le Théorème \ref{thm:1} à partir d'une version généralisant la série
\eqref{eq:fischer} de Fischer: \cite{burnoldigamma} explique la méthode mais
ne donne que deux termes, car il faut les considérations plus avancées de
\cite{burnolasymptotic} pour aller commodément plus loin.

Ici nous démontrerons le Théorème \ref{thm:1} d'une manière totalement
différente, en utilisant comme point de départ \cite{burnolirwin}.  Ceci est
l'occasion d'obtenir certains résultats sur les $v_m$ vues comme fonctions
analytiques de la base $b$, et de les employer.

\section{Remarques préliminaires}

%
%
Nous utilisons \cite[Thm. 4]{burnolirwin} (avec le
paramètre de ``niveau'' $\ell$ mis à la valeur $1$).  Ce théorème nous dit
dans le cas particulier $d=b-1$:
\begin{equation}\label{eq:K}
  K(b,b-1) = \sum_{d=1}^{b-2}\sum_{m=1}^\infty\frac{v_m}{(d+1)^{m+1}}
\end{equation}
\begin{rem}
  Pour $b=2$, $d=1$, on a $K(2,1)=0$ puisqu'il n'y a
  pas d'entier non nul n'utilisant pas le chiffre $1$ en écriture binaire, ce
  qui est compatible avec la formule ci-dessus, la somme sur $d$ étant alors
  vide.  Les quantités $v_m$ sont aussi définies pour $b=2$, elles sont dans
  ce cas toutes égales à $2$.
\end{rem}

Ces $v_m$ (qui ne dépendent ici que de $b$) vérifient les récurrences suivantes:
\begin{equation}\label{eq:vm}
  \forall m\geq0\quad (b^{m+1}-b+1)v_m = b^{m+1}+\sum_{j=1}^m\binom{m}{j}\gamma_jv_{m-j}
\end{equation}
où les symboles $\gamma_j$ sont définis par
  $\forall j\geq0\quad \gamma_j = \sum_{d=1}^{b-1} d^j$.
On fera attention que $\gamma_j$ est utilisé ici pour désigner ce qui était
noté $\gamma_j'$ dans \cite[Thm. 4]{burnolirwin}.  Les $v_m$ vérifient
\emph{aussi} une généralisation à toute base $b$ de la récurrence
\eqref{eq:fischerrec}, voir \cite{burnoldigamma}.  Mais nous ne l'utiliserons
pas.  La quantité $d' =
b-1-(b-1)$ de \cite[Thm. 4]{burnolirwin} vaut ici $0$. L'équation
\eqref{eq:vm} est aussi valable pour $m=0$ en considérant que la somme est
vide donc nulle.  Elle donne $v_0=b$.  On notera que $\gamma_0 = b-1$ et qu'on
a parfois avantage à écrire \eqref{eq:vm} avec $(b-1)v_m$ déplacé à droite.

Dans la suite nous allons traiter les $\gamma_j$ comme des polynômes de degré
$j+1$ en la variable $b$.  Rappelons la formule connue avec les nombres de
Bernoulli, valable pour tout $j\geq1$ entier (mais pas ici pour $j=0$ puisque
notre $\gamma_0$ vaut $b-1$):
\begin{equation*}
  \gamma_j = \sum_{p=0}^j \binom{j}{p}B_p\frac{b^{j+1-p}}{j+1-p}
\end{equation*}

Nous passerons bientôt de la variable $b$ à la variable $c=b^{-1}$, et
considérerons les $\gamma_j$ comme des fractions rationnelles en $c$.  Dans ce
cadre nous pourrons nous ramener à des polynômes en multipliant par $c^{j+1}$,
c'est-à-dire en considérant $c^{j+1}\gamma_j$.  On a en effet:
\begin{equation}\label{eq:bernoulli}
  c^{j+1}\gamma_j =  \sum_{p=0}^j \binom{j}{p}B_p\frac{c^p}{j+1-p} =
  \frac{1}{j+1} - \frac{c}{2} + \frac{j c^2}{12} - \frac{j(j-1)(j-2)c^4}{720} + \dots + B_j c^j
\end{equation}
De plus lorsque $c=1/b$, $b\in\NN$, $b\geq2$, on a  l'encadrement connu
\begin{equation}
  \label{eq:id2}
  \frac1{j+1} - \frac c2 \leq c^{j+1} \gamma_j \leq \frac1{j+1} -\frac c2 +
  \frac{j c^2}{12}
\end{equation}
qui résulte de la comparaison avec l'intégrale de la fonction $t\mapsto t^j$
(comme l'on pratique au début du développement d'Euler-MacLaurin par exemple).
On prendra garde qu'on ne peut pas étendre en général ces inégalités à la
variable réelle $c\in[0,\frac12]$ ni même l'encadrement ``au cran d'avant'' $0\leq
(j+1)c^{j+1} \gamma_j\leq 1$.


Nous aurons également besoin du résultat du calcul suivant, d'abord mené pour
$c=1/b$, $b\in\NN$, $b\geq2$. Soit $m$ un entier naturel:
\begin{align*}
  \sum_{j=0}^m \binom{m}{j}\frac{\gamma_j}{m+1-j}
&=\sum_{j=0}^m \frac1{m+1}\binom{m+1}{j}\sum_{d=1}^{b-1}d^j\\
&=\frac1{m+1}\sum_{d=1}^{b-1}((d+1)^{m+1}-d^{m+1}) = \frac{b^{m+1} - 1}{m+1}
\end{align*}
Donc pour $c=1/b$:
\begin{equation}\label{eq:id1}
   c^{m+1}\sum_{j=0}^m \binom{m}{j} \frac{\gamma_j}{m+1-j} = \frac{1-c^{m+1}}{m+1}
\end{equation}
Comme l'équation \eqref{eq:id1} compare deux fonctions polynomiales en la
variable $c$ et dit qu'elles coïncident en tous les $c=1/b$, $b\geq2$ entier,
c'est donc une identité pour toutes les valeurs complexes de
$c$.


\section{Estimation de moments}

Afin d'aider à suivre les démonstrations qui suivent, nous donnons ici
les débuts des développements en puissances de $c = b^{-1}$ au voisinage de
zéro, des $v_m$ pour $1\leq m \leq 6$:
\begin{align*}
  v_1&
  \begin{aligned}[t]
    &=\frac b2 + 1 + \frac{c}2\frac{1-2c}{1-c+c^2} \\
   &= \frac b2 +1 + \frac 12 c - \frac12 c^2 - c^3 + O(c^4)
  \end{aligned}
  \\
v_2 &
\begin{aligned}[t]
  &= \frac b3 + 1 +  \frac{c}{6}\frac{6-5c-7c^2+10c^3-6c^4}{(1-c+c^2)(1-c^2+c^3)}\\
  &= \frac b3 + 1 + c + \frac16 c^2 - c^3+ O(c^4)
\end{aligned}
\\
v_3 &= \frac b4 + 1 + c + \frac12 c^2 + 0 c^3 + O(c^4)\\
v_4 &= \frac b5 + 1 + c + \frac12 c^2 + \frac12 c^3 +O(c^4)\\
v_5 &= \frac b6 + 1 + c + \frac12 c^2 + \frac34 c^3 + O(c^4)\\
v_6 &= \frac b7 + 1 + c + \frac12 c^2 + c^3 + O(c^4)
\end{align*}
\begin{prop}
  La fraction rationnelle $v_m$ en la variable complexe $c=b^{-1}$, définie
  pour $m\geq0$ par les identités récurrentes \eqref{eq:vm}, possède un pôle
  simple en $c=0$ de résidu $(m+1)^{-1}$.  Elle n'a pas d'autres pôles dans le
  disque ouvert de rayon $\rho$ avec
$\rho\approx\np{0.755}$
  la racine réelle
  de l'équation $\rho^2(1+\rho) = 1$.  Et:
\begin{equation}\label{eq:taylor}
  m\geq 4\implies v_m = \frac{1}{(m+1)c} + 1  + c + \frac12 c^2 + \frac{m-2}{4} c^3 + O_{c\to0}(c^4) 
\end{equation}
\end{prop}
\begin{proof}
  On considère les fractions rationnelles $w_m = v_m - \frac{b}{m+1}$.  Alors,
  par \eqref{eq:vm} et en utilisant initialement une somme pour $j$ allant de $0$ à $m$ au
  lieu de $1$ à $m$ (on rappelle que $\gamma_0 = b-1$):
  \begin{align*}
\frac{b^{m+2}}{m+1} + b^{m+1} w_m &= b^{m+1} + \sum_{j=0}^m\binom{m}{j}\frac{b\gamma_j}{m+1-j} 
    +  \sum_{j=0}^m\binom{m}{j} \gamma_j w_{m-j}\\
\frac{1}{(m+1)c} + w_m &= 1 + \frac1c\sum_{j=0}^m\binom{m}{j}c^{m-j}\frac{c^{j+1}\gamma_j}{m+1-j}
   +  \sum_{j=0}^m\binom{m}{j} c^{j+1}\gamma_j c^{m-j} w_{m-j}\\
\frac{1}{(m+1)c} + w_m &= 1 +  \frac{1-c^{m+1}}{(m+1)c} +  \sum_{j=0}^m\binom{m}{j} c^{j+1}\gamma_j c^{m-j} w_{m-j}
  \end{align*}
  On a utilisé \eqref{eq:id1}.  De plus on a $w_0 = b - b = 0$ et $c \gamma_0
  = 1 -c$.  Donc au final, pour $m\geq1$ (si $m=1$ la somme est vide donc
  nulle):
\begin{equation}\label{eq:wm}
  (1 - c^m + c^{m+1})w_m = 1 -(m+1)^{-1} c^m + \sum_{j=1}^{m-1}\binom{m}{j} c^{j+1}\gamma_j c^{m-j} w_{m-j}
\end{equation}
En particulier $w_1 = \frac{1 - c/2}{1 - c + c^2}$ a ses pôles sur le cercle
unité.  Le polynôme $1 - c^2 + c^3$ a une seule racine réelle $-\rho$ avec
$\rho$ vérifiant $\rho^2(1+\rho)=1$ et valant un peu plus de $3/4$.  Les deux
autres racines complexes conjuguées sont en dehors du disque unité.  Lorsque
$z$ est un complexe avec $|z|<\rho$ et $m\geq2$ on a $|z^m-z^{m+1}|\leq
|z|^2(1+|z|)<\rho^2(1+\rho) = 1$.  Donc les pôles des $1 -c^m + c^{m+1}$ sont
tous en dehors du disque ouvert centré en l'origine de rayon $\rho$.  Les
fractions rationnelles $w_m$ sont donc des fonctions analytiques dans ce
disque ouvert, ce qui montre que $v_m$ comme fonction de $c$ a un pôle simple
en zéro de résidu $1/(m+1)$.

On suppose dorénavant que $m$ est au moins $4$.  On sépare les contributions
de $j=m-1$, $j=m-2$, $j=m-3$ à \eqref{eq:wm}.
On obtient
\begin{equation*}
  (1 - c^m + c^{m+1})w_m = 1 - (m+1)^{-1} c^m  
   + c^{m+1}\Bigl( m \gamma_{m-1} w_1 + \binom{m}2 \gamma_{m-2} w_2 + \binom{m}{3} \gamma_{m-3} w_3\Bigr)
  + \Gamma
\end{equation*}
avec $\Gamma$ la somme de $j=1$ à $m-4$ des $c^{j+1}\gamma_jc^{m-j}w_{m-j}$ et
$m-j\geq4$. Ce terme $\Gamma$ ne contribue pas aux 4 premiers coefficients de
Taylor du terme de droite de l'équation ci-dessus, pas plus que
$(m+1)^{-1}c^m$.  Diviser par $1 - c^m + c^{m+1}$ ne changera pas non plus ce
début de développement de Taylor $A(m) + B(m) c + C(m) c^2 + D(m) c^3$.  On a
immédiatement $A(m)=1$ puisque $c^{m+1}\gamma_{m-i}$ pour $i=1$, $2$, $3$ est
un polynôme nul en zéro.  D'après \eqref{eq:bernoulli}, les coefficients de
$c \cdot c^{m}m\gamma_{m-1} \cdot w_1$ jusqu'à $c^3$ seront polynomiaux en $m$ de
degrés au plus $2$.  De même $m(m-1) c^2 c^{m-1}\gamma_{m-2}w_2$ contribue lui
aussi de manière au plus quadratique en $m$.  Et $m(m-1)(m-2) c^3
c^{m-2}\gamma_{m-3} w_3$ ne contribue qu'un multiple de $m(m-1)$ au terme en
$c^3$.  Il en résulte que $A(m)$, $B(m)$, et $C(m)$ sont tous trois des
polynômes en $m$ au pire quadratique.  Mais avec un logiciel de calcul formel
ou à la main on vérifie les développements qui ont été donnés plus haut et qui
donnent $A(m)=B(m)=1$ pour $m=4,5,6$, et $C(m) = (m-2)/4$ pour $m = 4, 5, 6$.
D'où \eqref{eq:taylor}.
\end{proof}
\begin{thm}\label{thm:estimation}
  Il existe une constante absolue $K$ (effective) telle que:
  \begin{equation*}
    m\geq4, b\in\NN, b\geq2, c=b^{-1} \implies
    \left|v_m - \frac b{m+1} - 1 - c - \frac12 c^2 - \frac{m-2}4 c^3 \right|\leq K
      m^2 c^4
  \end{equation*}
\end{thm}
\begin{proof}
  On s'intéresse maintenant à l'écart $z_m$ entre $w_m=v_m - ((m+1)c)^{-1}$ et
  les quatre premiers termes du développement en puissances de $c$.
On définit, y-compris pour $0\leq m<4$, $z_m$ par:
\begin{equation*}
  w_m =  1 + c + \frac12 c^2 + \frac{m-2}{4} c^3  + z_m
\end{equation*}
On a par \eqref{eq:taylor} une majoration pour chaque $m\geq4$ en $K(m) c^4$
pour $0\leq c\leq \frac12$ par exemple.

Le problème est l'uniformité en $m$.  Nous l'établissons uniquement pour les
$c$ de la forme $1/b$, $b$ entier au moins $2$.  Supposons donc $m>4$ et qu'on
ait établi une majoration de $|z_p|$ en $K(p)c^4$ pour $4\leq p < m$ pour une
certaine fonction $K(p)$, pour ces $c$.

On obtient la récurrence des $z_m$ à partir de \eqref{eq:wm} (où l'on rappelle $w_0=0$):
\begin{align*}
    w_m &= 1 -  \frac{c^{m}}{m+1} +  c^{m+1}\sum_{j=0}^{m-1}\binom{m}{j} \gamma_j w_{m-j}\\
 1 + c + \frac12 c^2 + \frac{m-2}{4} c^3  +z_m&=1 -  \frac{c^{m}}{(m+1)}  
 \begin{aligned}[t]
   &+ (1 + c + \frac12 c^2-\frac12 c^3) U(m)
   + \frac{c^3}{4} V(m)\\
   &+
   c^{m+1}\sum_{j=0}^{m-1}\binom{m}{j} \gamma_j z_{m-j}
 \end{aligned}
\end{align*}
avec (en utilisant le même raisonnement qui a donné \eqref{eq:id2}):
\begin{equation*}
  U(m) = c^{m+1}\sum_{j=0}^{m-1}\binom{m}{j} \gamma_j =
  c^{m+1}\sum_{d=1}^{b-1} ((d+1)^m-d^m)=c^{m+1}(b^m-1) =c(1 - c^m) 
\end{equation*}
et
\begin{equation*}
  \begin{split}
    V(m) = c^{m+1}\sum_{j=0}^{m-1} (m-j)\binom{m}{j} \gamma_j =
    mc^{m+1}\sum_{j=0}^{m-1} \binom{m-1}{j} \gamma_j 
   \\ =  m c^{m+1}(b^{m-1}+\gamma_{m-1}-1)=mc^2+c\cdot mc^{m}\gamma_{m-1}-mc^{m+1}  
  \end{split}
\end{equation*}
Donc
\begin{equation}\label{eq:zmtaylor}
  (1 - c^m + c^{m+1}) z_m = 
  \begin{aligned}[t]
    &\frac12 c^2 + \frac{4-m}{4} c^3\\
    &-c^{m+1}(1 + c + \frac12 c^2-\frac12 c^3) -  \frac{c^{m}}{(m+1)}-\frac12 c^4+\frac{c^3}{4}V(m)\\
    &+ c^{2} m c^{m-1}\gamma_{m-1} z_1 + c^{3}\binom{m}2 c^{m-2}\gamma_{m-2}z_2
    +c^{m+1}\sum_{j=1}^{m-3}\binom{m}{j} \gamma_j z_{m-j}
  \end{aligned}
\end{equation}
On a pour nos $c$ et $m\geq5$: $0<V(m)< m c^2 + c \leq (\frac12 m +
1)c$, qui est multiplié par $c^3/4$, donc une majoration en $B_0 m c^4$
avec $B_0$ une constante absolue.  Les autres termes de la ligne intermédiaire
donnent une contribution majorée par $A_0 c^4$ pour une certain
$A_0$, pour $m\geq5$.  Donc en tout une majoration en $C_0 m c^4$.

Considérons ensuite la dernière somme, toujours pour $m\geq5$, mais en
s'arrêtant à $j=m-4$.  On peut majorer cette sous-somme, en valeur absolue,
par
\begin{equation*}
  (\max_{4\leq p < m} K(p)c^4) \cdot c^{m+1}\sum_{j=1}^{m-4} \binom{m}{j}\gamma_j
\end{equation*}
Pour la contribution de $j=m-3$, il se trouve que $z_3$ possède un zéro triple
à l'origine.  Donc il existe une constante absolue $C_3$ telle que $|z_3|\leq
C_3 |c|^3$ sur le disque de rayon $\frac12$.  En particulier ça vaut pour nos
$c$ inverses d'entiers, et comme $c|z_3|\leq C_3 c^4$ on peut combiner ce
terme à la somme déjà majorée en mettant en facteur non plus $c^{m+1}$ mais
$c^m$ et donc avec $K_3(m) = \frac12\max_{4\leq p < m} K(p) + C_3$ (on fait
ici la somme plutôt qu'un maximum car c'est plus simple pour ajouter les
autres termes d'erreur par la suite) on obtient une majoration:
\begin{equation*}
  K_3(m)c^4 c^m \sum_{j=1}^{m-3} \binom{m}{j}\gamma_j
  \begin{aligned}[t]
    &\leq 
    K_3(m)c^4 c^m \sum_{j=0}^{m-1} \binom{m}{j}\gamma_j< K_3(m) c^4
  \end{aligned}
\end{equation*}
Regardons maintenant le terme $c^{m+1}m(m-1)\gamma_{m-2} z_2$.  Il se trouve
que $z_2$ a un zéro double à l'origine.  Et on sait que $c^{m-1} (m-1)
\gamma_{m-2}$ est positif et majoré par $1$ (pour nos $c$!).  Donc on obtient ici une
majoration $C_2 m c^4$ avec une certaine constante absolue $C_2$.

En ce qui concerne le terme $c^{m+1}m \gamma_{m-1} z_1$ on procède plus
finement.  Ce terme contribue réellement au développement de Taylor de $z_m$
jusqu'à l'ordre $3$ inclus.  En effet $z_1 = -\frac12 c  - c^2 + O(c^3)$.
Et d'après \eqref{eq:id2} on sait que
\begin{equation*}
  c^{m}m\gamma_{m-1} = 1 -  \frac m2 c  + \theta \frac{m(m-1)}{12} c^2, \quad 0\leq \theta\leq 1
\end{equation*}
On attire l'attention du lecteur sur le fait que ceci serait faux si l'on
autorisait à $c$ des valeurs autres que celles spécifiées dans l'énoncé.

Ainsi nous obtenons du produit $c^{m+1}m\gamma_{m-1}z_1$ tout d'abord
$cz_1(1 - \frac m2 c)$ puis un terme supplémentaire qui est
majoré par un $C_1 m^2 c^4$ pour $C_1$ une constante absolue.
On a de plus avec une autre constante absolue $C_1'$:
\begin{equation*}
\left| cz_1(1 - \frac m2 c) -( -\frac12 c^2(1+\frac{4-m}2 c))\right| \leq C_1' m c^4
\end{equation*}
En effet on peut évaluer  par la formule de Taylor avec reste de
Lagrange qui fera intervenir une dérivée quatrième sur l'intervalle
$[0,\frac12]$ et comme la fonction considérée $cz_1- \frac m2 z_1c^2$ dépend
linéairement de $m$, on a le résultat ci-dessus.

Ceci donne donc une contribution $-\frac12 c^2 -\frac{4-m}4 c^3 +
O(mc^4)$ qui vient opportunément compenser les seuls termes qui subsistaient
dans \eqref{eq:zmtaylor} à l'ordre trois inclus.

Toutes les erreurs cumulées donnent un multiple $K_0(m)c^4$ avec $K_0(m)$
obtenu en additionnant à la moitié du maximum des $K(p)$ précédents une
quantité $T+Um+Vm^2$ avec des constantes absolues $T,U,V$, que l'on peut donc
englober dans un $Wm^2$: pour $m\geq5$, $K_0(m)= Wm^2 + \frac 12 \max_{4\leq p<m} K(p)$.

Mais il reste à prendre en compte le facteur multiplicateur $(1
-c^m(1-c))^{-1}$.  Pour $m\geq5$ et $c\in[0,\frac12]$ on peut le majorer par
$64/63$.  Ainsi une définition récurrente convenable de $K(m)$ est:
\begin{equation}\label{eq:Km}
  m\geq5\implies K(m) = \frac{64}{63}(\frac12\max_{4\leq p <m} K(p) + W m^2)
\end{equation}
Soit maintenant $Z$ une constante choisie pour qu'à la fois $|z_4|\leq 16 Z
c^4$ est vrai pour $|c|\leq \frac12$ et par ailleurs $W\leq \frac{31}{64} Z$.
Alors les quantités $K(m)$ définies par $K(4)=16 Z$ et la récurrence
\eqref{eq:Km} pour $m\geq5$ vérifient $K(m)\leq Zm^2$.  Et nous avons prouvé
$|z_m|\leq K(m)c^4$ pour $c=1/b$, $b\in\NN$, $b\geq2$.  Ceci conclut la
preuve.
\end{proof}

\section{Fin de la démonstration}

On passe maintenant à la démonstration du Théorème \ref{thm:1}.  On a donc
d'après \eqref{eq:K} et le Théorème \ref{thm:estimation}:
\begin{equation*}
  K(b,b-1) = \sum_{d=1}^{b-2}\frac b{d+1}
  \begin{aligned}[t]
    &+ (v_1 - \frac12 b - 1 - b^{-1} - \frac12 b^{-2} + \frac14 b^{-3})\sum_{d=1}^{b-2}\frac{1}{(d+1)^2}\\
    &+ (v_2 - \frac13 b - 1 - b^{-1} - \frac12 b^{-2}  )\sum_{d=1}^{b-2}\frac{1}{(d+1)^{3}}\\
    &+ (v_3 - \frac13 b - 1 - b^{-1} - \frac12 b^{-2}  - \frac14 b^{-3})\sum_{d=1}^{b-2}\frac{1}{(d+1)^{4}}\\
    &+ \sum_{d=1}^{b-2}\sum_{m=1}^\infty \frac{b(m+1)^{-1} + 1 + b^{-1} + \frac12 b^{-2} + \frac{m-2}4 b^{-3}+O(m^2 b^{-4})}{(d+1)^{m+1}}
  \end{aligned}
\end{equation*}
Comme $ \sum_{m=1}^\infty m ^2 (d+1)^{-m-1} = d^{-2}+2d^{-3} $ les termes en $O$ donnent
une contribution totale qui est $O(b^{-4})$. Le terme initial et les séries
infinies donnent pour chaque $d$:
\begin{equation*}
  - b \log(1-\frac1{d+1}) + (1+ b^{-1} + \frac12 b^{-2})\frac{1}{d(d+1)} 
  + \frac14 b^{-3} \left(\frac1{d^2}- \frac2{(d+1)d}\right)
\end{equation*}
et en sommant sur les $d\in\{1,\dots,b-2\}$ cela donne
\begin{equation*}
  b \log(b-1) + (1+ b^{-1} + \frac12 b^{-2} - \frac12 b^{-3})(1 - \frac1{b-1})  +\frac1{4b^3}\sum_{d=1}^{b-2}\frac1{d^2}
    =b\log(b) - \frac1{2b} - \frac{11}{6b^2} -\frac{13-\zeta(2)}{4b^3}+ O(b^{-4})
\end{equation*}
Par ailleurs on a
\begin{align*}
  v_1 -\frac12 b - 1 - \frac1b - \frac1{2b^2} + \frac1{4b^3} 
&= -\frac1{2b} - \frac1{b^2} - \frac3{4b^3} + O(b^{-4})\\
  v_2 - \frac13 b - 1 - \frac1b - \frac1{2b^2} 
&=-\frac1{3b^2} -\frac1{b^3} + O(b^{-4})\\
v_3 - \frac14 b - 1 - \frac1b - \frac1{4b^3}
&= -\frac1{4b^3} + O(b^{-4})
\end{align*}
Donc à ce stade nous avons au total
\begin{equation*}
  \begin{aligned}[t]
    b\log(b) &- \frac1{2b} - \frac{11}{6b^2} - \frac{13-\zeta(2)}{4b^3} + O(b^{-4}) \\
    &-\Bigl(\frac1{2b} + \frac1{b^2} + \frac3{4b^3} + O(b^{-4})\Bigr)\Bigl(\zeta(2)-1-\sum_{n=b}^\infty\frac 1{n^2}\Bigr)\\
    &-\Bigl(\frac1{3b^2} + \frac1{b^3} + O(b^{-4})\Bigr)(\zeta(3) - 1 + O(b^{-2}))\\
    &-\frac1{4b^3} (\zeta(4) - 1 + O(b^{-3}))
  \end{aligned}
\end{equation*}
Et par ailleurs $\sum_{n=b}^\infty\frac
1{n^2}=\frac1{b-\frac12} + O(b^{-3}) = b^{-1} + \frac12{b^{-2}} + O(b^{-3})$.

Après avoir tenu compte de tous ces termes, la formule \eqref{eq:main} du Théorème
\ref{thm:1} apparaît.

%
%
La constante implicite est effective et on pourrait donner une majoration
explicite de l'erreur.  Dans la pratique voici quelques résultats numériques,
les calculs des sommes de Kempner étant faits soit avec le code de
Baillie \cite{baillie1979,baillie2008} soit avec celui de l'auteur
\cite{burnolkempner,burnolirwin} qui utilise la formule \eqref{eq:K} ou plutôt
par défaut la variante de \cite{burnolkempner} donnant une série alternée
spéciale.
\[\def\arraystretch{1.5}
\begin{array}{r|l|l|l}
b&10&100&1000\\\hline\hline
                            b\log(b)&\np{23.025850930}&\np{460.517018599}&\np{6907.755278982137}\\
                b\log(b)-\frac{c_1}b&\np{22.943604227}&\np{460.508793928}&\np{6907.754456515104}\\
b\log(b)-\frac{c_1}b-\frac{c_2}{b^2}&\np{22.923148030}&\np{460.508589367}&\np{6907.754454469484}\\
b\log(b)-\frac{c_1}b-\frac{c_2}{b^2}-\frac{c_3}{b^3}&
                                    \np{22.920852925}&\np{460.508587071}&\np{6907.754454467189}\\
\hline
K&\np{22.920676619}&\np{460.508587055}&\np{6907.754454467187}
\end{array}
\]


%




\bigskip
\NoAutoSpacing

\singlespacing
\providecommand\bibcommenthead{}
\def\blocation#1{}

\end{document}